\documentclass{article}

\usepackage{amsmath,amssymb,amsthm,amsfonts,mathrsfs}

\def\R{\mathbb R}
\def\N{\mathbb N}
\def\S{\mathbb S}

\numberwithin{equation}{section}

\newtheorem{theorem}{Theorem}[section]
\newtheorem{lemma}[theorem]{Lemma}
\newtheorem{proposition}[theorem]{Proposition}

\newtheorem{definition}[theorem]{Definition\rm}

\newcommand{\e}{\varepsilon}

\DeclareMathOperator{\diam}{diam}   

\newcommand{\K}{\mathcal{K}}
\newcommand{\C}{\mathcal{C}}

\title{\textbf{Regularity properties of the distance functions to conjugate and cut loci
                    for viscosity solutions of Hamilton-Jacobi equations and applications in Riemannian geometry}}

\author{M.~Castelpietra\footnote{Universit\'e de Nice-Sophia
    Antipolis, Labo.\ J.-A.\ Dieudonn\'e, UMR 6621, Parc
    Valrose, 06108 Nice Cedex 02, France ({\tt
      castelpietra@unice.fr})}
\and
L.~Rifford\footnote{Universit\'e de Nice-Sophia
    Antipolis, Labo.\ J.-A.\ Dieudonn\'e, UMR 6621, Parc
    Valrose, 06108 Nice Cedex 02, France ({\tt
      rifford@unice.fr})}}

\date{}

\makeindex
\begin{document}

\maketitle

\begin{abstract}
Given a continuous viscosity solution of a Dirichlet-type Hamilton-Jacobi equation, we show that the distance function to the conjugate locus which is associated to this problem is locally semiconcave on its domain. It allows us to  provide a simple proof of the fact that  the distance function to the cut locus associated to the problem is locally Lipschitz on its domain. This result, which was already an improvement of a previous one by Itoh and Tanaka \cite{it01}, is due to Li and Nirenberg \cite{ln05}. Finally, we give applications of our results in Riemannian geometry. Namely, we show that the distance function to the conjugate locus on a Riemannian manifold is locally semiconcave. Then, we show that if a Riemannian manifold is a $C^4$ small perturbation of the round sphere, then all its tangent nonfocal domains are strictly uniformly convex.
\end{abstract}

\section{Introduction}

\subsection{}

Let $H: \R^n \times \R^n \rightarrow \R$ (with $n\geq 2$) be an Hamiltonian of
class $C^{k,1}$ (with $ k\geq 2$) which satisfies the three
following conditions:
\begin{itemize}
    \item[(H1)] (Uniform superlinearity) For every $K\geq 0$, there is $C(K)< \infty$ such that
    $$
    H(x,p) \geq K |p| - C(K) \qquad \forall (x,p)\in \R^n \times \R^n.
    $$
    \item[(H2)] (Strict Convexity in the adjoint variable) For every $(x,p) \in \R^n \times \R^n$,
    the second derivative $\frac{\partial^2 H}{\partial p^2} (x,p)$ is positive definite.
    \item[(H3)] For every $x\in \R^{n}, H(x,0)<0$.
\end{itemize}
Let $\Omega$ be an open set in $\R^n$ with compact boundary,
denoted by $S=\partial \Omega$, of class $C^{k,1}$. We are interested in the viscosity solution of the following Dirichlet-type
Hamilton-Jacobi equation
\begin{eqnarray}
\label{HJ}
\left\{ \begin{array}{rl}
H(x,du(x)) = 0, & \qquad \forall x \in \Omega,\\
u(x) = 0, & \qquad \forall x \in \partial \Omega.
\end{array} \right.
\end{eqnarray}

We recall that if $u:\Omega \rightarrow \R$ is a continuous function, its \textit{viscosity subdifferential}
 at $x\in \Omega$ is the convex subset of $\R^n$ defined by
$$
D^-u(x) := \left\{ d\psi (x) \mid  \psi \in C^{1}(\Omega) \mbox{ and }
  u-\psi \mbox{ attains a global minimum at } x\right\},  
$$
while its \textit{viscosity superdifferential} at $x$ is the convex subset of $\R^n$ defined by
$$
D^+u(x) := \left\{ d\phi (x) \mid \phi \in C^{1}(\Omega) \mbox{ and }
  u-\phi \mbox{ attains a global maximum at } x\right\}.
$$
Note that if $u$ is differentiable at $x\in \Omega$, then $D^-u(x) = D^+u(x) =\{du(x)\}$. A continuous function  $u:\Omega \rightarrow \R$ is said to be a \textit{viscosity subsolution} of $H(x,du(x))$ on $\Omega$ if the following property is satisfied:
$$
H(x,p) \leq 0 \qquad \forall x\in U, \quad \forall p \in D^+u(x).
$$
Similarly, a continuous function $u:\Omega \rightarrow \R$ is a said to be a \textit{viscosity supersolution} of $H(x,du(x))$ on $\Omega$ if 
$$
H(x,p) \geq 0 \qquad \forall x\in U, \quad \forall p \in D^- u(x).
$$
A  continuous function $u:\Omega \rightarrow \R$ is called a \textit{viscosity solution} of (\ref{HJ}) if it satisfies the boundary condition $u=0$ on $S$, and if it is both a viscosity subsolution and a viscosity supersolution of $H(x,du(x))=0$ on $\Omega$. The purpose of the present paper is first to study the distance functions to the cut and conjugate loci  associated with the (unique) viscosity solution of (\ref{HJ}).

\subsection{} 

The Lagrangian $L: \R^{n} \times \R^{n} \rightarrow \R$ which is
associated to $H$ by Legendre-Fenchel duality is defined by,
$$
L(x,v) := \max_{p\in
\R^{n}} \left\{ \langle p,v\rangle -H(x,p) \right\} \qquad \forall (x,v) \in \R^{n} \times \R^{n}.
$$
It is of class $C^{k,1}$ (see \cite[Corollary A.2.7 p.
287]{cs04}) and satisfies the properties of
uniform superlinearity and strict convexity in $v$. For every $x,y
\in \overline{\Omega}$ and $T\geq 0$, denote by  $\Omega_T (x,y)$ the set of
locally Lipschitz curves $\gamma : [0,T] \rightarrow
\overline{\Omega}$ satisfying $\gamma(0)=x$ and $\gamma (T)=y$. Then, set 
$$
l(x,y) := \inf \left\{ \int_0^T L(\gamma(t),\dot{\gamma}(t))dt \
\vert \ T \geq 0, \gamma \in \Omega_T (x,y)\right\}.
$$
The viscosity solution of (\ref{HJ}) is unique and can be characterized as follows:

\begin{proposition}
\label{PROP1}
The function $u:\overline{\Omega} \rightarrow \R$ given by
\begin{eqnarray}
\label{DEFu}
u(x) := \inf \left\{ l(y,x) \ \vert \ y \in \partial \Omega \right\}, \quad \forall x \in \overline{\Omega},
\end{eqnarray}
is well-defined and continuous on $\overline{\Omega}$. Moreover,
it is the unique viscosity solution of (\ref{HJ}).
\end{proposition}

The fact that $u$ is well-defined and continuous is easy and left
to the reader. The fact that the function $u$ given by (\ref{DEFu}) is a viscosity
solution of (\ref{HJ}) is a standard result in viscosity theory
(see \cite[Theorem 5.4 p. 134]{lions82}). The fact that, thanks to
(H3), $u$ is indeed the unique viscosity solution is less
classical; we refer the reader to \cite{bp88,ishii87} for its proof.

\subsection{}

Before giving in the next paragraph a list of properties satisfied by the viscosity solution of (\ref{HJ}), we recall some notions of nonsmooth analysis. \\

A function $u:\Omega \rightarrow \R$ is called \textit{locally semiconcave} on $\Omega$ if for every $\bar{x} \in \Omega$, there exist  $C,\delta >0$ such that
\begin{equation*}
\mu u(y)+ (1-\mu) u(x) -u(\mu x +(1-\mu) y) \leq \mu (1-\mu) C |x-y|^2,
\end{equation*}
for all $x,y$ in the open ball $B (\bar{x},\delta) \subset \Omega$ and every $\mu \in [0,1]$. Note that every locally semiconcave function is locally Lipschitz on its
domain, and thus, by Rademacher's Theorem, is differentiable almost everywhere on its domain. A way to prove that a given function  $u:\Omega \rightarrow \R$ is locally semiconcave on $\Omega$ is to show that, for  every $\bar{x} \in \Omega$, there exist a
 $\sigma, \delta >0$  such that, for every  $x\in B(\bar{x},\delta) \subset \Omega$,
    there is $p_{x} \in \R^n$ such that
    \begin{equation*}
    u(y) \leq u(x) + \langle p_{x}, y-x\rangle + \sigma |y-x|^2 \qquad \forall y \in B(\bar{x},\delta).
    \end{equation*}
We refer the reader to \cite{rifford08,riffordbook} for the proof of this fact. \\

If $u:\Omega \rightarrow \R$ is a continuous function, its \textit{limiting subdifferential} at $x\in \Omega$ is the subset of $\R^n$ defined by
$$
\partial_{L} u(x) := \left\{ \lim_{k \rightarrow \infty} p_{k} \mid p_{k}
\in D^-u(x_{k}), \,x_{k} \rightarrow x \right\}.
$$
By construction, the graph of the limiting subdifferential is closed
in $\R^n \times \R^n$. Moreover, the function $u$ is locally Lipschitz on $\Omega$ if and only if the graph of the  limiting subdifferential of $u$ is locally bounded (see \cite{clsw98,riffordbook}). \\

Let $u:\Omega \rightarrow \R$ be a locally Lipschitz function. The \textit{Clarke generalized differential} (or simply \textit{generalized gradient}) of
$u$ at the point $x\in \Omega$ is the nonempty compact convex subset of $\R^n$ defined by
$$
\partial u(x) := \mbox{conv} \left(\partial_L u(x)\right),
$$
that is, the convex hull of the limiting subdifferential of $u$ at $x$. Notice that, for every $x\in \Omega$,
$$
D^-u(x) \subset \partial_{L}u(x) \subset \partial u(x) \qquad \mbox{and} \qquad D^+u(x) \subset \partial u(x).
$$
It can be shown that, if $\partial u(x)$ is a singleton, then  $u$ is differentiable at $x$ and $\partial u(x) =\{du(x)\}$. The converse result is false. \\

Let $u:\Omega \rightarrow \R$ be a function which is locally semiconcave on $\Omega$.  It can be shown (see \cite{cs04,riffordbook}) that for every $x\in \Omega$ and every $p\in D^+u(x)$, there are $C,\delta>0$ such that 
$$
 u(y) \leq u(x) + \langle p,y-x\rangle + \frac{C}{2} |y-x|^2 \qquad \forall y \in B(x,\delta) \subset \Omega,
$$
In particular, $D^+u(x) = \partial u(x)$ for every $x\in \Omega$. The \textit{singular set}
of $u$ is the subset of $\Omega$ defined by  
\begin{align*}
\Sigma(u)  := & \left\{ x\in \Omega \mid u \mbox{ is not differentiable at } x \right\} \\
= & \left\{ x\in \Omega \mid \partial u(x) \mbox{ is not a singleton} \right\} \\
= & \left\{ x\in \Omega \mid \partial_L u(x) \mbox{ is not a singleton} \right\} .
\end{align*}
From Rademacher's theorem, $\Sigma (u)$ has Lebesgue measure zero. In fact, the following result holds (see \cite{cs04,riffordbook}):

\begin{theorem}
\label{THMsingset}
Let $\Omega$ be an open subset of $M$. The singular set of a locally
semiconcave function $u:\Omega \rightarrow \R$ is countably $(n-1)$-rectifiable, \textit{i.e.}, is contained in a countable union of locally Lipschitz hypersurfaces of $\Omega$.
\end{theorem}

As we shall see, the Li-Nirenberg Theorem (see Theorem \ref{THMcut}) allows to prove that $\Sigma(u)$ has indeed finite $(n-1)$-dimensional Hausdorff measure. 

\subsection{}

From now on, $u:\Omega \rightarrow \R$ denotes the unique viscosity solution of (\ref{HJ}).  Let us collect some properties satisfied by $u$:
\begin{itemize}
\item[(P1)] The function $u$ is locally semiconcave on $\Omega$.
\item[(P2)] The function $u$ is $C^{k,1}$ in a neighborhood of $S$ (in $\Omega$).
\item[(P3)] The function $u$ is $C^{k,1}$ on the open set $\Omega \setminus \overline{\Sigma(u)}$.
\item[(P4)] For every $x\in \Omega$ and every $p\in \partial_L
u(x)$, there are $T_{x,p} >0$ and a curve $\gamma_{x,p} :
[-T_{x,p},0] \rightarrow \R$ such that $\gamma_{x,p}(-T_{x,p}) \in
S$ and, if $(x,p):[-T_{x,p},0] \rightarrow \R^n \times \R^n$
denotes the solution to the Hamiltonian system
\begin{eqnarray*}
\left\{ \begin{array}{ccc}
\dot{x}(t) & = & \frac{\partial H}{\partial p} (x(t),p(t))\\
\dot{p}(t) & = & - \frac{\partial H}{\partial x}(x(t),p(t))
\end{array}
\right.
\end{eqnarray*}
with initial conditions $x (0) =x, p(0)=p$, then we have
$$
\gamma_{x,p} (t) = x(t) \quad  \mbox{and}  \quad du
(\gamma_{x,p}(t)) = p(t), \quad \forall t \in [-T_{x,p},0],
$$
which implies that
$$
u(x) - u(\gamma_{x,p}(t)) = \int_t^{0} L\left(
\gamma_{x,p}(s),\dot{\gamma}_{x,p}(s)\right) ds, \quad \forall t
\in [-T_{x,p},0].
$$
\item[(P5)] For every $T>0$ and every locally Lipschitz curve $\gamma
:[-T,0] \rightarrow \overline{\Omega}$ satisfying $\gamma(0)=x$,
$$
u(x) - u(\gamma(-T)) \leq \int_{-T}^{0} L\left( \gamma (s),\dot{\gamma} (s)\right) ds.
$$
\item[(P6)] As a consequence, we have for every $x\in \Omega$,
every $p\in \partial_L u(x)$, every $T>0$, and every locally Lipschitz
curve $\gamma :[-T,0] \rightarrow \overline{\Omega}$ satisfying
$\gamma(0)=x$ and $\gamma(-T) \in \partial \Omega$,
$$
\int_{-T_{x,p}}^{0} L\left(
\gamma_{x,p}(t),\dot{\gamma}_{x,p}(t)\right) dt \leq \int_{-T}^{0}
L\left( \gamma (s),\dot{\gamma} (s)\right) ds.
$$
\item[(P7)] If $x\in \Omega$ is such that $u$ is $C^{1,1}$ in a neighborhood of $x$, then for every $t< 0$, the function $u$ is $C^{1,1}$ in a neighborhood of $\gamma_{x,p}(t)$ (with $p=du(x)$). 
\end{itemize}

The proof of (P1) can be found in \cite{rifford08}. Properties (P2)-(P3) are straighforward consequences of the method of characteristics (see \cite{cs04}). Properties (P4)-(P6) taken together give indeed a characterization of the fact that $u$ is a viscosity solution of (\ref{HJ}) (see for instance \cite{fathibook,riffordbook}). Finally the proof of (P7) can be found in \cite{rifford08}.

\subsection{}

We proceed now to define the exponential mapping associated to our Dirichlet problem. Let us denote by $\phi_t^H$ the Hamiltonian flow acting on $\R^n \times \R^n$. That is, for every $x,p \in \R^n \times \R^n$, the function $t \mapsto \phi^H_t (x,p)$ denotes the solution to  
\begin{eqnarray}
\label{sysH}
\left\{ \begin{array}{ccc}
\dot{x}(t) & = & \frac{\partial H}{\partial p} (x(t),p(t))\\
\dot{p}(t) & = & - \frac{\partial H}{\partial x}(x(t),p(t))
\end{array}
\right.
\end{eqnarray}
satisfying the initial condition $\phi_0^H(x,p)=(x,p)$. Denote by $\pi:\R^n \times \R^n \rightarrow \R^n$ the projection on the first coordinates $(x,p)\mapsto x$.  The \textit{exponential} from $x \in S$ in time $t\geq 0$ is defined as  
$$
\exp(x,t) := \pi \left( \phi_t^H (x,du(x)) \right).
$$
Note that, due to blow-up phenomena, $\exp(x,t)$ is not necessarily defined for any $t\geq 0$. For every $x\in S$, we denote by $T(x)\in (0,+\infty)$ the maximal positive time such that $\exp(x,t)$ is defined on $[0,T(x))$. The function $(x,t)\mapsto \exp(x,t)$ is of class $C^{k-1}$ on its domain.

\begin{definition}
For every $x\in S$, we denote by $t_{conj}(x)$, the first time $t\in (0,T(x))$ such that $d \exp(x,t)$ is singular. The function $t_{conj}:S \rightarrow (0,+\infty) \cup \{+\infty\}$ is called the distance function to the conjugate locus. The set of $x\in S$  such that $t_{conj}(x)<\infty$ is called the domain of $t_{conj}$.
\end{definition}

Note that, if  $d \exp(x,t)$ is nonsingular for every $t\in (0,T(x))$, then $t_{conj}(x)=+\infty$. Furthermore, by (H3), we always have $\frac{\partial}{\partial t}\exp(x,t) \neq 0$. Therefore, it could be shown that for any $x \in S$ and any $t\in (0,T(x))$, $d\exp(x,t)$ is singular if and only if  $\frac{\partial  \exp}{\partial x} (x,t)$ is singular.

\begin{theorem}
\label{THMconjLIP}
 Assume that $H$ and $S=\partial
\Omega$ are of class $C^{2,1}$. Then, the domain of $t_{conj}$ is open and the function $x\mapsto
t_{conj}(x)$ is locally Lipschitz on its domain.
\end{theorem}

If $M$ is a submanifold of $\R^n$ of class at least $C^2$, a function $u:M \rightarrow \R$ is called locally semiconcave on $M$ if for every $x \in M$
there exist a neighborhood $\mathcal{V}_x$ of $x$ and a diffeomorphism
$\varphi_x: \mathcal{V}_x \rightarrow \varphi_x(\mathcal{V}_x) \subset \R^n$ of class $C^2$ such that
$f\circ \varphi_x^{-1}$ is locally semiconcave on the open set $\varphi_x(\mathcal{V}_x) \subset \R^n$.  

\begin{theorem}
\label{THMconjLSC}
Assume that $H$ and $S=\partial
\Omega$ are of class $C^{3,1}$. Then, the function $x\mapsto
t_{conj}(x)$ is locally semiconcave on its domain.
\end{theorem}

The proofs of Theorems \ref{THMconjLIP} and \ref{THMconjLSC} are postponed to Section \ref{SecProofConj}. Applications of these results in Riemannian geometry are given in Section \ref{SecApp}. The strategy that we will develop to prove the above theorems will allows us to show that any tangent nonfocal domain of a $C^4$-deformation of the round sphere $(\S^n,g^{can})$ is strictly uniformly convex, see Section \ref{SecApp}.

\subsection{}

The \textit{cut-locus} of $u$ is defined as the closure of its singular set,
that is
$$
\mbox{Cut} (u) = \overline{\Sigma(u)}.
$$

\begin{definition}
For every $x \in S$, we denote by $t_{cut}(x) >0$, the first time
$t\in (0,T(x))$ such that $\exp(x,t) \in \mbox{Cut}(u)$. The function $t_{cut}:S \rightarrow (0,+\infty)$ is called the distance function to the cut locus.
\end{definition}

Note that the following result holds. 

\begin{lemma}
\label{LEMconjcut}
For every $x\in S$, $t_{cut}(x)$ is finite and $t_{cut}(x) \leq t_{conj}(x)$.
\end{lemma}

\noindent \textit{Proof of Lemma \ref{LEMconjcut}.} Let $x\in S$ be fixed; let us prove that $t_{cut}(x)$ is finite. Suppose that $\exp(x,t) \notin \mbox{Cut}(u)$ for all $t\in (0,T(x))$. Two cases may appear. If there is $t\in (0,T(x))$ such that $\exp (x,t) \notin \overline{\Omega}$, then  this means that there is $\bar{t} \in (0,T(x))$ such that $\exp(x,\bar{t})\in S$. So, thanks to (P3), $u$ is $C^{k,1}$ along the curve $\gamma(\cdot)$ defined as $\gamma(t):=\exp (x,t)$ for $t\in [0,\bar{t}]$. Thanks to (P4), we have
$$
0= u(\gamma(\bar{t}) - u(\gamma(0)) = \int_0^{\bar{t}} L(\gamma(s),\dot{\gamma}(s)) ds.
$$
But by definition and (H3), the Lagrangian $L$ satisfies for every $(x,v) \in \R^n \times \R^n$,
\begin{eqnarray*}
L(x,v) & := & \max_{p\in \R^{n}} \left\{ \langle p,v\rangle -H(x,p) \right\} \\
& \geq & -H(x,0) >0,
\end{eqnarray*}
which yields 
$$
 \int_0^{\bar{t}} L(\gamma(s),\dot{\gamma}(s)) ds >0.
$$
So, we obtain a contradiction. If $\exp(x,t)$ belongs to $\overline{\Omega}$ for all $t\in (0,T(x))$, this means, by compactness of $\overline{\Omega}$, that $T(x)=+\infty$. So, thanks to (P3) and (P4), setting $\gamma(t) := \exp(x,t)$ for any $t\geq 0$, we obtain 
$$
u(\gamma(t)) = u(\gamma(t)) -u(\gamma(0)) =\int_0^t L(\gamma(s),\dot{\gamma}(s)) ds \qquad \forall t \geq 0.
$$
But, by compactness of $\overline{\Omega}$, on the one hand there is $\rho>0$ such that $ L(\gamma(s),\dot{\gamma}(s)) \geq \rho$ for any $t\geq 0$ and on the other hand $u$ is bounded from above. We obtain a contradiction. Consequently, we deduce that there is necessarily $t\in (0,T(x))$ such that $\exp (x,t)\in \mbox{Cut}(u)$, which proves that $t_{cut}(x)$ is well-defined. 

Let us now show that $t_{cut}(x) \leq t_{conj}(x)$. We argue by contradiction. Suppose that $t_{conj}(x) <t_{cut}(x)$. Thanks to (P3), this means that the function $u$ is at least $C^{1,1}$ in an open neighborhood $\mathcal{V}$ of $\bar{y}:=\exp (x,t_{conj}(x))$ in $\Omega$. Set for every $y\in \mathcal{V}$, 
$$
T(y) := \inf \left\{ t \geq 0 \ \vert \ \phi^{H}_{-t} (y,du(y) ) \in S \right\}.
$$
By construction, one has $T(\bar{y})=t_{conj}(x)$. Moreover since the curve $t \mapsto \exp(x,t)$ is transversal to $S$ at $t=0$, taking $\mathcal{V}$ smaller if necessary, we may assume that $T$ is of class $C^{k-1,1}$ on $\mathcal{V}$. Define $F:\mathcal{V} \rightarrow S$ by 
$$
F(y):= \pi \left( \phi_{-T(y)}^H(y,du(y)\right) \qquad \forall y\in \mathcal{V}.
$$
The function $F$ is Lipschitz on $\mathcal{V}$ and satisfies $\exp (F(y),T(y))=y$ for every $y\in \mathcal{V}$. This show that the function $\exp$ has a Lipschitz inverse in a neighborhood of the point $(x,t_{conj}(x))$. This contradicts the fact that $d \exp (x,t_{conj}(x)) $ is singular. \hfill $\Box$ \\

Actually, the distance function to the cut locus at $x\in S$ can be seen as the time after which the "geodesic" starting at $x$ ceases to be minimizing.

\begin{lemma}
\label{ceases}
For every $x\in S$, the time $t_{cut}(x)$ is the maximum of times $t\geq 0$ satisfying the following property:
\begin{eqnarray}
\label{eqceases}
u(\exp(x,t)) = \int_0^t L\left( \exp(x,s), \frac{\partial \exp}{\partial t}(x,s)\right) ds.
\end{eqnarray}
\end{lemma}

\noindent \textit{Proof of Lemma \ref{LEMconjcut}.} Set $T:= t_{cut}(x)$. First, by (P4), we know that 
$$
u(\exp(x,T)) = \int_0^T L\left( \exp(x,s), \frac{\partial \exp}{\partial t}(x,s)\right) ds.
$$
Argue by contradiction and assume that there is $\bar{t}>T$ such that (\ref{eqceases}) is satisfied. By (P7), for every $s\in [T,\bar{t}]$, the point $\exp(x,s)$ necessarily belongs to $\mbox{Cut}(u)$ (the fact that $\exp(x,s)$ belongs to $\Omega$ is a consequence of the proof of Lemma \ref{LEMconjcut}). Fix $\bar{s} \in (T,\bar{t})$ and set $\bar{y}:= \exp(x,\bar{s})$. Two cases may appear: either $\bar{y}$ belongs to $\Sigma(u)$ or $\bar{y}$ belongs to $\mbox{Cut}(u) \setminus \Sigma(u)= \overline{\Sigma(u)} \setminus \Sigma(u)$. By (P4), if $\bar{y}$ belongs to $\Sigma(u)$, then there is a curve $\bar{\gamma}_p:[-T_p,0] \rightarrow \overline{\Omega}$ with 
\begin{eqnarray}
\label{corner}
\dot{\bar{\gamma}}_{p}(0)\neq \frac{ \partial \exp}{\partial t}(x,\bar{s})
\end{eqnarray}
 such that
$$
u(\bar{y}) = \int_{-T_p}^0 L \left(\bar{\gamma}_{p}(s),\dot{\bar{\gamma}}_{p}(s)\right) ds.
$$
Thanks to (P4)-(P6), this means that the curve $\tilde{\gamma}:[-T_p,\bar{t}-\bar{s}] \rightarrow \overline{\Omega}$ defined as 
$$
\tilde{\gamma}(s) := \left\{ \begin{array}{rcl}
\bar{\gamma}_{p} (s) & \mbox{ if } & s \in [-T_{p},0] \\
\exp(x,\bar{s} + s) & \mbox{ if } & s \in [0,\bar{t}-\bar{s}],
\end{array}
\right.
$$
minimizes the quantity 
$$
\int_{-T_p}^{\bar{t}-\bar{s}} L( \gamma(s),\dot{\gamma}(s)) ds,
$$
among all curves $\gamma :[-T_p,\bar{t}-\bar{s}] \rightarrow \overline{\Omega}$ such that $\gamma(-T_p)=\bar{\gamma}(-T_p)$ and $\tilde{\gamma}(\bar{t}-\bar{s}) = \exp(x,\bar{t})$. But, thanks to (\ref{corner}), the curve $\tilde{\gamma}$ has a corner at $s=0$. This contradicts the regularity of minimizing curves given by Euler-Lagrange equations. Therefore, we deduce that $\bar{y}$ necessarily belongs to $\mbox{Cut}(u) \setminus \Sigma(u)$. This means that $u$ is differentiable at $\bar{y}$ and that there is a sequence of points $\{y_k\}$ of $\Sigma(u)$ converging to $\bar{y}$. Thus by (P4)-(P6), for each $k$, there are $p_k^1\neq p_k^2$ in $\partial_L u(y_k)$ and $T_k^1,T_k^2>0$ such that  
$$
y_k = \exp \left( \phi_{-T_k^1}^H (y_k,p_k^1) \right) = \exp \left( \phi_{-T_k^2}^H (y_k,p_k^2) \right).
$$
Since the sequences $\{p_k^1\}, \{p_k^2\}$ and $\{T_k^1\}, \{T_k^2\}$ necessarily converge to $du(\bar{y})$ and $\bar{s}$, we deduce that $\exp$ is singular at $(x,\bar{s})$. To summarize, we proved that if there is $\bar{t}>T$ such that (\ref{eqceases}) is satisfied, then for every $s\in [T,\bar{t}]$, the function $\exp$ is singular at $(x,s)$. Let us show that it leads to a contradiction\footnote{The fact that a "geodesic" ceases to be minimizing after the first conjugate is well-know. However, since our Lagrangian (or equivalently our Hamiltonian) is merely $C^{2,1}$ (and indeed for sake of completeness), we prefer to provide the proof of this fact.}. Using the notations which will be defined later in Section \ref{secSYMP}, there is $(h,v)\neq 0\in U(x)$ such that the solution $(h(\cdot),v(\cdot))$ of the linearized Hamiltonian system (\ref{HamSysLin}) starting at $(h,v)$ satisfies $h(T)=0$. Moreover, since any $s\in [T,\bar{t}]$ is a conjugate time, there is indeed a sequence $\{s_k\}$ converging to $T$ associated to a sequence of vectors $\{(h_k,v_k)\}$ converging to $(h,v)$ such that each solution $(h_k(\cdot),v_k(\cdot))$ of (\ref{HamSysLin}) starting at $(h_k,v_k)$ satisfies $h_k(s_k)=0$. Since the Hamiltonian flow preserves the canonical symplectic form $\sigma$, one has for any $k$,
$$
\langle h_k(T),v(T)\rangle =0.
$$
But since $h_k(s_k)=0$, the differential equation (\ref{HamSysLin}) yields  
$$
h_k(T)= - (s_k-T) Q(x,s_k) v_k(s_k) +o (s_k-T).
$$
Since $\{v_k(s_k)\}$ converges to $v(T)$, we deduce that 
$$
\langle Q(x,T) v(T),v(T)\rangle =0,
$$
which contradicts the fact that $Q(x,T)= \frac{\partial^2 H}{\partial p^2}\left(\bar{y},\frac{\partial L}{\partial v} (\bar{y}, \frac{ \partial \exp}{\partial t}(x,T)\right)$ is positive definite. \hfill $\Box$\\

Define the set $\Gamma(u) \subset \mbox{Cut}(u)$ as 
$$
\Gamma (u) := \left\{ \exp(x,t) \ \vert \ x\in S, t>0 \mbox { s.t. } t=t_{conj}(x)=t_{cut}(x) \right\},
$$

The two above lemmas yields the following result.

\begin{lemma}
One has 
$$
\mbox{Cut}(u) = \Sigma (u) \cup \Gamma (u).
$$
\end{lemma}

The following theorem is due to Li and Nirenberg \cite{ln05}; we provide a new proof of it in Section \ref{sec3}.

\begin{theorem}
\label{THMcut} 
Assume that $H$ and $S=\partial \Omega$
are of class $C^{2,1}$. Then the function $x \mapsto t_{cut}(x)$
is locally Lipschitz on its domain.
\end{theorem}

As a corollary, as it is done in \cite{ln05}, since 
$$
\mbox{Cut} (u) = \left\{ \exp(x,t_{cut}(x)) \ \vert \ x \in S\right\},
$$
we deduce that the cut-locus of $u$ has a finite $(n-1)$-dimensional Hausdorff measure. Note that it can also be shown (see \cite{cs04,frv,pignotti02}) that, if $H$ and $S=\partial \Omega$ are of class $C^{\infty}$, then  the set $
\Gamma (u)$ has Hausdorff dimension less or equal than $n-2$. 




\section{Proofs of Theorems \ref{THMconjLIP} and \ref{THMconjLSC}} \label{SecProofConj}

\subsection{Proof of Theorem \ref{THMconjLIP}}\label{secSYMP}

Before giving the proof of the theorem, we recall basic facts in symplectic geometry. We refer the reader to \cite{am78,cannas01} 
for more details. \\

The symplectic canonical form $\sigma$ on $\R^n \times \R^n$ is
given by
$$
\sigma \left( \left( \begin{array}{c} h_1\\ v_1 \end{array}
\right) ,  \left( \begin{array}{c} h_2 \\ v_2 \end{array} \right)
\right)= \left\langle \left( \begin{array}{c} h_1\\ v_1 \end{array}
\right) , J \left( \begin{array}{c} h_2 \\ v_2 \end{array}
\right)\right\rangle,
$$
where $J$ is the $2n \times 2n$ matrix defined as
$$
J = \left( \begin{array}{cc}
0_n & I_n  \\
-I_n & 0_n
\end{array}
\right).
$$
It is worth noticing that any Hamiltonian flow in $\R^n \times \R^n$ preserves the symplectic form. That is, if $(x(\cdot),p(\cdot))$ is a trajectory of (\ref{sysH}) on the interval $[0,T]$, then for every $(h_1,v_1), (h_2,v_2) \in \R^n \times \R^n$ and every $t\in [0,T]$, we have
$$
\sigma \left( \left( \begin{array}{c} h_1\\ v_1 \end{array}
\right) ,  \left( \begin{array}{c} h_2 \\ v_2 \end{array} \right)
\right) = \sigma \left( \left( \begin{array}{c} h_1(t)\\ v_1(t) \end{array}
\right) ,  \left( \begin{array}{c} h_2(t) \\ v_2(t) \end{array} \right)
\right),
$$
where $(h_i(\cdot),v_i(\cdot))$ (with $i=1,2$) denotes the solution on $[0,T]$ to the linearized Hamiltonian system (see (\ref{HamSysLin}) below) along $(x(\cdot),p(\cdot))$ with initial condition $(h_i,v_i)$ at $t=0$.   
We recall that a vector space $J \subset \R^n \times \R^n$ is called \textit{Lagrangian} if it a $n$-dimensional vector space where the symplectic form $\sigma$ vanishes. If a $n$-dimensional vector subspace $J$ of $\R^n \times \R^n$ is transversal to the vertical subspace, that is $J\cup \{0\} \times \R^n =\{0\} $, then there is a $n\times n$ matrix $K$ such that 
$$
J = \left\{ \left( \begin{array}{c} h \\ Kh \end{array}
\right)  \ \vert \ h\in \R^n \right\}.
$$
It can be checked that $J$ is Lagrangian if and only if $K$ is a symmetric matrix. \\

Let $x\in S$ be fixed. Denote by $(x(\cdot),p(\cdot))$ the solution to the Hamiltonian system (\ref{sysH}) on $[0,T(x))$ satisfying $(x(0),p(0))=(x,du(x))$. The linearized  Hamiltonian system along $(x(\cdot),p(\cdot))$ is given by 
\begin{eqnarray}
\label{HamSysLin}
\left\{ \begin{array}{ccc}
\dot{h}(t) & = & B(x,t)^* h(t) + Q(x,t) v(t) \\
\dot{v}(t) & = & -A(x,t) h(t) - B(x,t) v(t),
\end{array}
\right.
\end{eqnarray}
where the matrices $A(x,t)$, $B(x,t)$ and $Q(x,t)$ are respectively given by
$$
\frac{\partial^2 H}{\partial x^2}(x(t),p(t)), \quad
\frac{\partial^2 H}{\partial x\partial p}(x(t),p(t)), \quad
\frac{\partial^2 H}{\partial p^2}(x(t),p(t)),
$$
and where $B(x,t)^*$ denotes the transpose of $B(x,t)$. Define the matrix
$$
M(x,t):=\left(
\begin{array}{cc}
B(x,t)^* & Q(x,t) \\
-A(x,t) & -B(x,t)\end{array}\right),
$$
and denote by $R(x,t)$ the $2n \times 2n$ matrix solution of
$$
\begin{cases}
\displaystyle \frac{\partial R}{\partial t}(x,t)=M(x,t)R(x,t)\\
R(x,0)=I_{2n}.
\end{cases}
$$
Finally, let us set the following spaces (for every $t\in [0,T(x))$):
\begin{eqnarray*}
J(x,t)& := &\left\{ R(x,t)^{-1} \left( \begin{array}{c} 0 \\ w
\end{array} \right) \ \vert \ w \in \R^n \right\},\\
U(x)&:=&\left\{ \left(\begin{array}{c}h\\
D^2u(x)h \end{array}\right) \ \vert \ h\in\R^n \right\}.
\end{eqnarray*}
The following result is the key tool in the proofs of Theorems \ref{THMconjLIP} and \ref{THMconjLSC}.

\begin{lemma} \label{lem:conj1}
The following properties hold:
\begin{itemize}
\item[(i)] The spaces $J(x,t)$ (for all $t\in (0,T(x))$) and $U(x)$ are Lagrangian subspaces of $\R^n \times \R^n$; moreover, one has 
 $$
t_{conj} (x) = \min \left\{ t \geq 0 \ \vert \ J(x,t) \cap U(x)
\neq \{0\} \right\}.
$$
\item[(ii)] For every $t \in (0,t_{conj}(x)]$, the space $J(x,t)$ is transversal to the vertical subspace, that is  
\begin{equation*}
J(x,t) \cap \left( \{0\} \times \R^n \right) = \{0\} \qquad
\forall t \in (0,t_{conj}(x)].
\end{equation*}
\item[(iii)] If we denote for every $t \in (0,t_{conj}(x)]$, by $K(x,t)$ the
symmetric matrix such that
$$
J(x,t) = \left\{
\left(\begin{array}{c}h\\K(x,t)h\end{array}\right) \ \vert \ h \in
\R^n \right\},
$$
then the mapping $t \in [0,T(x)) \mapsto K(x,t)$ is of class $C^{k-1,1}$. Moreover there is a continuous function $\delta >0$ which is defined on the domain of the exponential mapping such that 
$$
\dot{K} (x,t) := \frac{\partial }{\partial t} K(x,t) \geq \delta (x,t) I_n \qquad \forall t \in (0,T(x)).
$$
\end{itemize}
\end{lemma}

\begin{proof}
Let us prove assertion (i). The fact that $J(t,x)$ and $U(x)$ are Lagrangian subspaces of $\R^n \times \R^n$  is easy, its proof is left to the reader. Suppose that there exists
$$
0 \neq \left(\begin{array}{c}h \\v\end{array}\right)\in J(x,t)\cap U(x).
$$
On the one hand, for a solution of (\ref{HamSysLin}) with initial data
$(h, v)$, we have that $h(t)=0$, since $(h,v)$ is in $J(x,t)$. On
the other hand, since $(h,v)\in U(x)$, $d \exp(x,t)h=h(t)=0$ with
$h\neq 0$, i.e. $d\exp(x,t)$ is singular. Conversely, if $d \exp(x,t)$ is singular for some $t\in (0,T(x))$, then there is $h\neq 0$ such that $d \exp(x,t)h=0$. Then there exists $v(t)\in\R^n$ such that
$$
R(x,t)\left(\begin{array}{c}h\\D^2 u(x)h\end{array}\right)=
\left(\begin{array}{c}d
\exp(x,t)h\\v(t)\end{array}\right)=\left(\begin{array}{c}0\\v(t)\end{array}\right),
$$
that is, $(h,D^2u(x)h)\in J(x,t)\cap U(x)$.

Let us prove assertion (ii). We argue by contradiction and assume that there is $t \in (0,t_{conj}(x)]$ such that
$J(x,t) \cap \left( \{0\} \times \R^n\right) \neq \{0\}$. By definition of $t_{conj}(x)$, we deduce that  
\begin{eqnarray}
\label{eq9mai0}
J(x,s) \cap U(x) = \{0\} \qquad \forall s \in [0,t).
\end{eqnarray}
Doing a change of coordinates if necessary, we may assume that $D^2u(x)=0$, that is 
$$
U(x) =\R^n \times \{0\}.
$$
By (\ref{eq9mai0}), we know that, for every $s\in [0,t)$, $J(x,s)$ is a Lagrangian subspace which is transversal to $U(x)$ . Hence there is, for every $s\in [0,t)$, a symmetric $n\times n$ matrix $\mathcal{K}(s)$ such that 
\begin{eqnarray}
\label{eq9mai1}
J(x,s) =  \left\{ \left( \begin{array}{c} \mathcal{K}(s)v \\ v \end{array}
\right)  \ \vert \ v \in \R^n \right\}.
\end{eqnarray}
Let us use the following notation: we split any matrix $R$ of the form $2n\times 2n$ in four matrices
$n\times n$ so that
$$
R=\left( \begin{array}{cc} R_1 & R_2\\
                            R_3 & R_4 \end{array}\right).
$$
Indeed, for any fixed $w\in \R^n$ and any $s\in [0,t)$,
$$
R(x,s)^{-1}\left( \begin{array}{c} 0 \\
                                    w \end{array}\right)=
\left( \begin{array}{c} h_{w,s}\\
                        v_{w,s} \end{array}\right)=
\left( \begin{array}{c} \mathcal{K}(s) v_{w,s} \\
                        v_{w,s}\end{array}\right),
$$
where $h_{w,s}=\left(R(x,s)^{-1}\right)_2 w$ and
$v_{w,s}=\left(R(x,s)^{-1}\right)_4 w$. Thanks to (\ref{eq9mai0}), the matrix $\left(R(x,t)^{-1}\right)_4$
is non-singular for every $s\in(0,t)$, then we have 
$$
\mathcal{K}(s)=\left(R(x,s)^{-1}\right)_2 \left(R(x,s)^{-1}\right)_4^{-1}.
$$
This shows that the function $s \in [0,t) \mapsto \mathcal{K}(s)$ is if class $C^{k-1,1}$. We now proceed to compute the derivative of $\mathcal{K}$ at some $\bar{s}\in (0,t)$, that we shall denote by  $\dot{\mathcal{K}}(\bar{s})$. Let $v\neq 0\in \R^n$ be fixed, set $h_{\bar{s}}:= \mathcal{K}(\bar{s})v$ and consider the unique $w_{\bar{s}} \in \R^n$ satisfying
$$
R(x,\bar{s})\left(\begin{array}{c} h_{\bar{s}}\\
                v \end{array}\right)=
\left(\begin{array}{c} 0\\
               w_{\bar{s}} \end{array}\right) \qquad \forall s \in (0,t).
$$
Define the $C^1$ curve $\phi :(0,t) \rightarrow \R^n \times \times \R^n$ by  
$$
\phi(s)=\left(\begin{array}{c} h_s\\
                                    v_s \end{array}\right):=
R(x,s)^{-1}\left(\begin{array}{c} 0\\
                                        w_{\bar{s}} \end{array}\right) \qquad \forall s\in (0,t).
$$
The derivative of $\phi$ at $\bar{s}$ is given by
$$
\dot{\phi}(\bar{s}) = \frac{\partial}{\partial s} \left[ R(x,s)^{-1} \right] \left(\begin{array}{c} 0\\
 w_{\bar{s}} \end{array}\right) = - R(x,s)^{-1} M(x,s) \left(\begin{array}{c} 0\\
  w_{\bar{s}} \end{array}\right).
$$
 Thus,  since the Hamiltonian flow preserves the symplectic form, we have
\begin{eqnarray*}
\sigma(\phi(\bar{s}),\dot{\phi}(\bar{s})) &=& \sigma \left( R(x,\bar{s})^{-1} \left(\begin{array}{c} 0\\w_{\bar{s}} \end{array}\right),- R(x,\bar{s})^{-1}M(x,\bar{s})\left(\begin{array}{c}0\\w_{\bar{s}}
\end{array}\right)\right)\\
&=& \sigma  \left( \left(\begin{array}{c} 0\\w_{\bar{s}} \end{array}\right),- M(x,\bar{s})\left(\begin{array}{c}0\\w_{\bar{s}}
\end{array}\right)\right)\\
& = & \langle Q(x,\bar{s})w_{\bar{s}},w_{\bar{s}} \rangle.
\end{eqnarray*}
By construction, the vector $\phi (s)$ belongs to $J(x,s)$ for any $s\in (0,t)$. Hence, it can be written as
$$
 \phi(s)=\left(\begin{array}{c} h_s\\
v_s \end{array}\right) = \left(\begin{array}{c} \mathcal{K}(s) v_s\\
v_s \end{array}\right).
$$
Which means that
$$
\dot{\phi}(s) = \left(\begin{array}{c} \dot{\mathcal{K}}(s) v_s + \mathcal{K}(s) \dot{v}_s\\
\dot{v}_s \end{array}\right).
$$
Thus, we have (using that $v_{\bar{s}} =v$)
\begin{eqnarray*}
\sigma (\phi(\bar{s}),\dot{\phi}(\bar{s})) & = & \sigma \left( \left(\begin{array}{c} \mathcal{K}(\bar{s}) v\\
v \end{array}\right),  \left(\begin{array}{c} \dot{\mathcal{K}}(\bar{s}) v \\
\dot{v}_{\bar{s}} \end{array}\right) \right) + \sigma \left( \left(\begin{array}{c} \mathcal{K}(\bar{s}) v \\
v \end{array}\right),  \left(\begin{array}{c} \mathcal{K}(\bar{s}) \dot{v}_{\bar{s}} \\
0 \end{array}\right) \right) \\
& = & \left\langle  \left(\begin{array}{c} \mathcal{K}(\bar{s}) v \\
v \end{array}\right), \left(\begin{array}{c} \dot{v}_{\bar{s}}\\ -\dot{\mathcal{K}}(\bar{s}) v  \end{array}\right) \right\rangle + \left\langle \left(\begin{array}{c} \mathcal{K}(\bar{s}) v \\
v \end{array}\right),  \left(\begin{array}{c} 0 \\ - \mathcal{K}(\bar{s}) \dot{v}_{\bar{s}} \end{array}\right) \right\rangle \\
& = & \langle \mathcal{K}(\bar{s}) v, \dot{v}_{\bar{s}} \rangle - \langle v, \dot{\mathcal{K}}(\bar{s}) v \rangle - \langle  v, \mathcal{K}(\bar{s}) \dot{v}_{\bar{s}}  \rangle \\
& = &  - \langle v, \dot{\mathcal{K}}(\bar{s}) v \rangle,
\end{eqnarray*}
since $\mathcal{K}(\bar{s})$ is symmetric. Finally, we deduce that 
\begin{eqnarray}
\label{eq9mai33}
\langle v,\dot{\mathcal{K}}(\bar{s})v \rangle = -\langle w_{\bar{s}},  Q(x,\bar{s})w_{\bar{s}}\rangle <0.
\end{eqnarray}
By assumption, we know that $J(x,t) \cap \left( \{0\} \times \R^n \right) \neq \{0\}$, which can also be written as 
$$
J(x,t) \cap J(x,0) \neq \{0\}.
$$
This means that there is $v \neq 0$ and a sequence $\left\{\left(\begin{array}{c} h_k \\
v_k \end{array}\right)\right\}$ in $\R^n \times \R^n$ such that 
$$
\lim_{k \rightarrow \infty} \left(\begin{array}{c} h_k \\
v_k \end{array} \right) = \left(\begin{array}{c} 0 \\
v \end{array} \right) \quad \mbox{and} \quad \left(\begin{array}{c} h_k \\
v_k \end{array}\right) \in J(x,t-1/k) \quad \forall k \mbox{ large enough in } \N.
$$
But we have for any large $k\in \N$, $h_k=\mathcal{K}(t-1/k)v_k$. Hence we deduce that   $\lim_{k\rightarrow \infty} \mathcal{K}(t-1/k)v_k=0$. But, thanks to (\ref{eq9mai33}) we have for $k$ large enough
$$
\langle v_k, \mathcal{K}(t-1/k)v_k \rangle = \int_0^{t-1/k} \langle v_k, \dot{\mathcal{K}}(s)v_k \rangle ds \leq \int_0^{t/2}   \langle v_k, \dot{\mathcal{K}}(s)v_k \rangle ds.
$$
But 
$$
\lim_{k\rightarrow \infty} \int_0^{t/2}   \langle v_k, \dot{\mathcal{K}}(s)v_k \rangle ds =  \int_0^{t/2}   \langle v, \dot{\mathcal{K}}(s)v \rangle ds <0.
$$
This contradicts the fact that $\lim_{k\rightarrow \infty} \mathcal{K}(t-1/k) v_k=0$ and concludes the proof of assertion (ii). We note that another way to prove (ii) would have been to use the theory of Maslov index, see \cite{agrachev}.\\
It remains to prove (iii). By (ii),  for every $t\in (0,t_{conj}(x)]$, the matrix $\left(R(x,t)^{-1}\right)_2$ is nonsingular and the matrix $K(x,t)$ is given by
$$
K(x,t)=\left(R(x,t)^{-1}\right)_4 \left(R(x,t)^{-1}\right)_2^{-1}.
$$
This shows that the function $t\in (0,t_{conj}(x)] \mapsto K(x,t)$ is of class $C^{k-1,1}$. Let us compute $\dot{K}(x,t)$ for some $t \in (0,t_{conj}(x)]$. Let $h \in \R^n$ be fixed, set $v_t:= K(x,t)h$ and consider the unique $w_t \in \R^n$ satisfying
$$
R(x,t)\left(\begin{array}{c} h\\
                v_t \end{array}\right)=
\left(\begin{array}{c} 0\\
               w_t \end{array}\right),
$$
that is 
$$
w_t = \left[ R(x,t)_3 +R(x,t)_4\right] h.
$$
Define the  $\C^1$ curve $\varphi: (0,t_{conj}] \rightarrow \R^n \times \R^n$ by
$$
\varphi(s)=\left(\begin{array}{c} h_s\\
                                    v_s \end{array}\right):=
R(x,s)^{-1}\left(\begin{array}{c} h\\
                                        v_t \end{array}\right), \qquad \forall s\in (0,t_{conj}(x)].
$$
As above, on the one hand we have 
\begin{eqnarray*}
\sigma(\varphi(t),\dot{\varphi}(t)) &=& \sigma \left( R(x,t)^{-1} \left(\begin{array}{c} 0\\w_{t} \end{array}\right),- R(x,t)^{-1}M(x,t)\left(\begin{array}{c}0\\w_{t}
\end{array}\right)\right)\\
&=& \sigma  \left( \left(\begin{array}{c} 0\\w_{t} \end{array}\right),- M(x,t)\left(\begin{array}{c}0\\w_{t}
\end{array}\right)\right)\\
& = & \langle Q(x,t)w_{t},w_{t} \rangle.
\end{eqnarray*}
On the other hand, using the fact that $\varphi(s) \in J(x,s)$ for any $s$, we also have 
$$
\sigma(\varphi(t),\dot{\varphi}(t)) = \langle K(x,t) h, h \rangle.
$$
For every $t\in (0,t_{conj}(x)]$, the linear operator $:\Psi (x,t): h\mapsto w_t := \left[ R(x,t)_3 +R(x,t)_4\right] h$ is invertible. If we denote, for every $t\in (0,t_{conj}(x)]$, by $\lambda(x,t) >0$, the smallest eigenvalue of the symmetric   matrix $Q(x,t)$, then we have for any $h\in \R^n$,
\begin{eqnarray*}
\langle K(x,t) h , h\rangle = \langle Q(x,t)w_{t},w_{t} \rangle & \geq & \lambda (x,t) |w_t|^2 \\ 
& \geq & \lambda (x,t) \|\Psi(x,t)^{-1}\|^{-2} |h|^2.
\end{eqnarray*}
The function $\delta$ defined as 
$$
\delta (x,t) := \lambda (x,t) \|\Psi(x,t)^{-1}\|^{-2} \qquad \forall x\in S, \forall t\in (0,t_{conj}(x)],
$$
depends continuously on $(x,t)$. This concludes the proof of Lemma  \ref{lem:conj1}.
\end{proof}

We are now ready to prove Theorems \ref{THMconjLIP}.

\begin{proof}[Proof of Theorem \ref{THMconjLIP}]

Let $\bar{x}\in S$ such that $\bar{t} :=t_{conj}(\bar{x}) <\infty$ be fixed. By Lemma  \ref{lem:conj1}, there is $h \in \R^n$ with $|h|=1$ such that $K(\bar{x},\bar{t})h=D^2u(\bar{x})h$. There is $\rho>0$ such that the function $\Psi : \left( S \cap B(\bar{x},\rho) \right) \times (\bar{t}-\rho,\bar{t}+\rho) \rightarrow \R $ defined by 
\begin{eqnarray}
\label{Psi}
\Psi (x,t) := \langle \left[K(x,t) -D^2u(x)\right] h, h\rangle,
\end{eqnarray}
is well-defined (note that $\Psi(\bar{x},\bar{t}) =0$). The function $\Psi$ is locally Lipschitz in the $x$ variable and of class $C^{k-1,1}$ in the $t$ variable. Moreover, restricting $\rho$ if necessary, we may assume that 
$$
\frac{\partial \Psi }{\partial t} (x,t) = \langle \dot{K}(x,t) h,h\rangle \geq \delta (x,t) \geq \frac{1}{2} \delta (\bar{x},\bar{t})>0 \qquad \forall x \in S \cap B(\bar{x},\rho), \forall t \in  (\bar{t}-\rho,\bar{t}+\rho).
$$
Thanks to the Clarke Implicit Function Theorem (see \cite[Corollary p. 256]{clarke83}), there are an open neighborhood $\mathcal{V}$  of $\bar{x}$ and a Lipschitz function $\tau :\mathcal{V} \rightarrow \R$ such that 
$$
\Psi ( x,\tau (x)) =0 \qquad \forall x \in \mathcal{V}.
$$
This shows that for every $x\in \mathcal{V}$, $t_{conj}(x)$ is finite. To prove that $t_{conj}$ is locally Lipschitz on its domain, it suffices to show that  for every $\bar{x}$ in the domain of $t_{conj}$, there is a constant $K>0$ and an open neighborhood $\mathcal{V}$ of $\bar{x}$ such that for every $x\in \mathcal{V}$, there is a neighborhood $\mathcal{V}_x$ of $x$ in $S$ and a function $\tau_x :\mathcal{V}_x \rightarrow \R$ which is $K$-Lipschitz and which satisfies
$$
\tau_x(x) = t_{conj}(x) \quad \mbox{and} \quad t_{conj}(y) \leq \tau_x(y) \quad \forall y\in \mathcal{V}_x.
$$
In the proof above, the Lipschitz constant of $\tau$ depends only on the Lipschitz constant of $\Psi$ and on a lower bound on $\delta(\bar{x},\bar{t})$. The result follows.   
\end{proof}

\subsection{Proof of Theorem \ref{THMconjLSC}}

Let $\bar{x}\in S$ in the domain of $t_{conj}(x)$. By Lemma  \ref{lem:conj1}, there is $h \in \R^n$ with $|h|=1$ such that $K(\bar{x},\bar{t})h=D^2u(\bar{x})h$. There is $\rho>0$ such that the function $\Psi : \left( S \cap B(\bar{x},\rho) \right) \times (\bar{t}-\rho,\bar{t}+\rho) \rightarrow \R $ defined by (\ref{Psi}) is well-defined. Since $k\geq 3$, $\Psi$ is at least of class $C^{1,1}$. Moreover, $\Psi (\bar{x},\bar{t})=0$ and
$$
\frac{\partial \Psi }{\partial t} (\bar{x},\bar{t}) = \langle \dot{K}(\bar{x},\bar{t}) h,h\rangle \geq \delta (\bar{x},\bar{t}) >0.
$$
By the usual Implicit Function Theorem, there exist a an open ball $\mathcal{B}$  of $\bar{x}$ and a $C^{1,1}$ function $\tau :\mathcal{B} \rightarrow \R$ such that 
$$
\Psi ( x,\tau (x)) =0 \qquad \forall x \in \mathcal{B}.
$$
This means that we have 
$$
\tau (\bar{x}) = t_{conj}(\bar{x}) \quad \mbox{and} \quad t_{conj}(x) \leq \tau(x) \qquad \forall x\in \mathcal{B}.
$$
Moreover, derivating $\Psi(x,\tau (x))=0$ yields
$$
\nabla \tau (x) = -\frac{\frac{\partial \Psi}{\partial x}(x,\tau(x))}{\frac{\partial \Psi}{\partial t}(x,\tau(x)} \qquad \forall x \in \mathcal{B}.
$$
This shows that the Lipschitz constant of $\nabla \tau$ as well as the radius of $\mathcal{B}$ are controlled by the Lipschitz constants of $\frac{\partial K}{\partial x}$ and $D^2u$. This proves that $t_{conj}(x)$ is locally semiconcave on its domain. 

\section{Proof of Theorem \ref{THMcut}}\label{sec3}

We have to show that there is $L>0$ such that the following property holds:
\begin{itemize}
\item[(PL)] For every $x\in S$, there are a neighborhood $\mathcal{V}_x$ of $x\in S$ and a $L$-Lipschitz function $\tau_x: \mathcal{V}_x \rightarrow \R$ satisfying
$$
\tau_x (x)=t_{cut}(x) \quad \mbox{and} \quad t_{cut}(y)\leq \tau_x (y) \qquad \forall y \in \mathcal{V}_x. 
$$
\end{itemize}
First, we claim that $t_{cut}$ is continuous on $S$. Let $x\in S$ be fixed and $\{x_k\}$ be a sequence of  points in $S$ converging to $x$ such that $t_{cut}(x_k)$ tends to $T$ as $k$ tends to $\infty$. Since a sequence of "minimizing curves" is still minimizing, we know by Lemma \ref{ceases} that $t_{cut}(x)\geq T$. But each point $\exp(x_k,t_{cut}(x_k))$ belongs to $\mbox{Cut}(u)$. So, since $\mbox{Cut}(u)$ is closed, the point $\exp(x,T)$ belongs to $\mbox{Cut}(u)$. This proves the continuity of $t_{cut}$.\\

Let $\hat{S}\subset S$ be the set defined by 
$$
\hat{S} := \left\{ x\in S \ \vert \ t_{conj}(x)=t_{cut}(x) \right\}.
$$
Since by continuity $t_{cut}$ is bounded, the set $\hat{S}$ is included in the domain of $t_{conj}$. Therefore, by Theorem \ref{THMconjLIP}, $\hat{S}$ is compact and there is $L_1>0$ such that $t_{cut}=t_{conj}$ is $L_1$-Lipschitz on $\hat{S}$ (in the sense of (PL)). \\

Let $\bar{x} \in S\setminus \hat{S}$ be fixed. Set $\bar{t}:=t_{cut}(\bar{x}), \bar{y}:=\exp (\bar{x},\bar{t}),$ and $(\bar{y},\bar{p}):=\phi_{\bar{t}}^H (\bar{x},du(\bar{x}))$. Since $\exp$ is not singular at $(\bar{x},\bar{t})$, one has 
$$
\diam \left(\partial u(\bar{y}))\right)=:\mu>0.
$$
This means that there is $x'\in S$ such that $\exp(x',t')=\bar{y}$ (with $t':=t_{cut}(x')$) and 
$$
|\bar{p}-p'|>\frac{\mu}{2},
$$
where $p'$ is defined by $(\bar{y},p')=\phi_{t'}^H (x',du(x'))$. Since $p'\in \partial u(\bar{y})=D^+u(\bar{y})$, by semiconcavity of $u$, there are $\delta,C>0$ such that 
\begin{eqnarray*}
u(y) \leq u(\bar{y}) + \langle p', y-\bar{y} \rangle + \frac{C}{2} |
y-\bar{y}|^2  \qquad \forall y \in B(\bar{y},\delta).
\end{eqnarray*}
Set $g(y):=u(\bar{y}) + \langle p', y-\bar{y} \rangle + C |
y-\bar{y}|^2$ for every $y\in B(\bar{y},\delta)$ and define the $C^1$  function $\Psi :S \times \R \rightarrow \R$ by
$$
\Psi (x,t) := g(\exp (x,t)) -\int_0^t L \left( \exp
(x,s),\frac{\partial \exp}{\partial s} (x,s) \right) ds.
$$
Note that $\Psi (\bar{x},\bar{t}) =0$. Moreover if $x\neq \bar{x}$ is such that $\exp(x,t) \in B(\bar{y},\delta)$ and $\Psi(x,t)=0$ for some $t>0$, then we have 
\begin{multline*}
u(\exp(x,t)) - \int_0^t  L \left( \exp (x,s),\frac{\partial \exp}{\partial s} (x,s) \right) ds \\
< g(\exp(x,t)) - \int_0^t  L \left( \exp (x,s),\frac{\partial \exp}{\partial s} (x,s) \right) ds=0.
\end{multline*}
Which means that $t_{cut}(x) \leq t$. Set for every $t\in [0,\bar{t}], \bar{\gamma}(t) := \exp (\bar{x},t)$. We have
\begin{eqnarray*}
\frac{\partial \Psi}{\partial t} (\bar{x},\bar{t}) & = & \langle p', \dot{\bar{\gamma}}(\bar{t}) \rangle  - L(\bar{\gamma} (\bar{t}),\dot{\bar{\gamma}}(\bar{t})) \\
& = & \langle p' -\bar{p}, \dot{\bar{\gamma}}(\bar{t}) \rangle  +
H(\bar{y},\bar{p}) = \langle p'-\bar{p},
\dot{\bar{\gamma}}(\bar{t})\rangle .
\end{eqnarray*}
Two cases may appear:\\

\noindent First case: there is $\rho>0$ such that $\mu\geq \rho$. Since the set $\{p\ \vert \ H(\bar{y},p)\leq 0)$ is uniformly convex, we deduce that the quantity 
$$
\frac{\partial \Psi}{\partial t} (\bar{x},\bar{t}) =  \langle p'-\bar{p},
\dot{\bar{\gamma}}(\bar{t})\rangle =  \langle p'-\bar{p}, \frac{\partial H}{\partial p} (\bar{y},\bar{p}) \rangle 
$$
is bounded from below by some constant $\epsilon(\rho)>0$. By the Implicit Function Theorem, there are an open ball $\mathcal{B}$ of $\bar{x}$ and a $C^1$ function $\tau: \mathcal{B} \cap S \rightarrow \R$ such that 
$$
\Psi (x,\tau(x)) =0 \qquad \forall x \in \mathcal{B} \cap S,
$$
where the Lipschitz constant of $\tau$ is bounded from above by $M/\epsilon(\rho)$, where $M$ denotes the Lipschitz constant of $\Psi$. This shows that there is $L_2>0$ such that $t_{cut}$ is $L_2$ -Lipschitz (in the sense of (PL)) on the set 
$$
S_{\rho} := \left\{ x \in S \ \vert \ \diam \left( \partial u (\exp (t_{cut}(x),x)) \right) \geq \rho \right\}.
$$
\\
\noindent Second case: $\mu$ is small enough. Without loss of generality, doing a global change of coordinates if necessary, we may assume that $S$ is an hyperplan in a neighborhood of $\bar{x}$ and that $D^2u(\bar{x})=0$. Set for every $s\in [0,\bar{t}]$,
$$
 \overline{L}_x(s):=\frac{\partial
L}{\partial x}(\bar{\gamma}(s),\dot{\bar{\gamma}}(s)), \quad \overline{L}_v(s):=\frac{\partial L}{\partial
v}(\bar{\gamma}(s),\dot{\bar{\gamma}}(s)),
$$
and 
$$
\overline{h}_{\nu}(s):= d\exp (\bar{x},s) ( \nu ) \qquad \forall \nu \in T_{\bar{x}} S \subset \R^n.
$$
Then
\begin{eqnarray*}
\langle \frac{\partial \Psi}{\partial x} (\bar{x},\bar{t}), \nu \rangle & = &
\langle p' , \frac{\partial \exp}{\partial
x}(\bar{x},\bar{t}) (\nu)
\rangle - \int_0^{\bar{t}} \langle \overline{L}_x(s), \overline{h}_{\nu} (s)\rangle  + \langle \overline{L}_v(s) , \dot{\overline{h}}_{\nu}(s) \rangle ds \\
& = & \langle p' , \frac{\partial \exp}{\partial
x}(\bar{x},\bar{t}) (\nu)
\rangle + \int_0^{\bar{t}} \langle \overline{L}_x(s) - \frac{d}{ds}\overline{L}_v(s), \overline{h}_{\nu} (s)\rangle  ds -   \left[ \langle \overline{L}_v(\cdot) , \overline{h}_{\nu}(\cdot) \rangle \right]_0^{\bar{t}} \\
& = & \langle p' , \frac{\partial \exp}{\partial x}(\bar{x},\bar{t}) (\nu) \rangle   -  \langle \overline{L}_v(\bar{t}) , \overline{h}_{\nu}(\bar{t}) \rangle \\
& = & \langle p' - \bar{p} ,  \overline{h}_{\nu}(\bar{t}) \rangle .
\end{eqnarray*}
Recall that $(\overline{h}_{\nu}(t),\bar{v}_{\nu}(t))$ is the solution of the linearized Hamiltonian
system (\ref{HamSysLin}) along $\bar{\gamma}$ starting at $\overline{h}_{\nu}(0)=\nu$ and $\overline{v}_{\nu}(0)=D^2u(\bar{x})\nu=0$. Let
us denote by $(h'(t),v'(t))$ the solution of (\ref{HamSysLin})  along $\bar{\gamma}$ such that
$h'(0)=x'-\bar{x}$ and $v'(0)=D^2u(\bar{x})(x'-x)=0$. Then, if $p' -\bar{p}$ is
small, $p' - \bar{p}$ equals $v'
(\bar{t})$ up to a quadratic term. But since the Hamiltonian flow  preserves the symplectic form, there is $D>0$ such that we have
for any $\nu\in T_{\bar{x}} S$ of norm one,
$$
\left|\langle \overline{h}_{\nu}(\bar{t}), v' (\bar{t})\rangle\right| = 
\left|\langle h' (\bar{t}) , \overline{v}_{\nu}(\bar{t})\rangle\right|
\leq D| x'-\bar{x}|^2,
$$
because\footnote{Just use Taylor's formula together with the fact that $\langle h'(\bar{t}),\dot{\bar{\gamma}}(\bar{t})\rangle =0$.} we know that $\exp (\bar{x},\bar{t}) = \exp (x',t')$. In conclusion, we have that $\frac{\partial \Psi}{\partial x}
(\bar{x},\bar{t})$ is bounded from above by $D' | x'-\bar{x}|^2$ for some $D'>0$. Besides, since $H(\bar{y},\bar{p})=H(\bar{y},p')=0$, we have,
by Taylor's formula,
$$
0 = \langle \frac{\partial H}{\partial p} (\bar{y},\bar{p}),
p'-\bar{p} \rangle + \frac{1}{2} \langle\frac{\partial^2
H}{\partial p^2}
(\bar{y},p)(p'-\bar{p}),p'-\bar{p}\rangle
$$
for some $p$ on the segment $[\bar{p},p']$. Therefore we
deduce that, for some $c>0$,
$$
\left| \frac{\partial \Psi}{\partial t}(\bar{z},\bar{t}) \right|
\geqslant c |p'-\bar{p}|^2,
$$
where we also have a positive constant $k$ such that
$|p'-\bar{p}| \geq k |x'-\bar{x}|$. Then, by the
Implicit Function Theorem, the function $\tau_{\bar{x}}(\cdot)$ is
well defined as the function such that
$\Psi(x,\tau_{\bar{x}}(x))=0$, and its gradient is bounded from above. This yields that if $\mu$ is taken small enough, then there there is $L_3$ such that $t_{cut}$ is $L_3$-Lipschitz on the set 
$$
S_{\rho}' := \left\{ x \in S \ \vert \ 0 < \diam \left( \partial u (\exp (t_{cut}(x),x)) \right) < \rho \right\}.
$$
This concludes the proof of Theorem \ref{THMcut}.

\section{Applications in Riemannian Geometry} \label{SecApp}

\subsection{}

Let $(M,g)$ be a smooth compact Riemannian manifold and $x\in M$ be fixed. The \textit{cut locus} of $x$, denoted by $\mbox{Cut}(x)$ is defined as the closure of the set of points $y$ such that there are at least two distinct minimizing geodesics between $x$ and $y$. The Riemannian distance to $x$, denoted by $d_g(x,\cdot)$, is locally semiconcave on $M\setminus \{x\}$. Then we have
$$
\mbox{Cut}(x) = \overline{\Sigma(d_g(x,\cdot)}.
$$
For every $v \in T_xM$, we denote by $\gamma_v$ the geodesic curve starting from $x$ with speed $v$. For every $v \in T_xM$, we set $\|v\|_x = g_x(v,v)$ and we denote by $\mathcal{S}_1^x$ the set of $v \in T_xM$ such that $\|v\|_x=1$. The \textit{distance function to the cut locus} (from $x$) $t_{cut}^x : \mathcal{S}_1^x \rightarrow \R$ is defined by
$$
t_{cut}^x (v) := \min \left\{ t \geq 0 \ \vert \ \gamma_v(t) \in \mbox{Cut}(x)\right\}.
$$
We prove easily that $t_{cut}^x$ is continuous on $\mathcal{S}_1^x$ (see \cite{sakai96}). 

\subsection{} 

Let $T^*M$ denote the cotangent bundle and $g$ be the cometric on $T^*M$, the Hamiltonian associated with $g$ is given by 
$$
H(x,p) = \frac{1}{2} \|p\|^2.
$$
For every $x\in M$, the Riemannian distance to $x$ which we denote from now by $d_g^x$ is a viscosity solution to the Eikonal equation
$$
H(x,du(x)) =\frac{1}{2} \qquad \forall x \in M \setminus \{x\}.
$$
The following result, due to Itoh and Tanaka \cite{it01}, can be seen (see \cite{riffordsard}) as a consequence of Theorem \ref{THMcut}.

\begin{theorem}
\label{THMRiem1}
The function $t_{cut}^x$ is Lipschitz on $\mathcal{S}_1^x$. 
\end{theorem}

We denote by $\exp_x: T_xM \rightarrow \R$ the Riemannian exponential mapping from $x$. Since $M$ is assumed to be compact, it is well-defined and smooth on $T_xM$. We recall that $\exp_x$ is said to be singular at $w\in T_xM$ if $d\exp_x(w)$ is singular.   The \textit{distance function to the conjugate locus} (from $x$) $t_{conj}^x:\mathcal{S}_1^x \rightarrow \R$ is defined by 
$$
t_{conj}^x (v) := \min \left\{  t \geq 0 \ \vert \ \exp_x(t) \mbox{ is singular}\right\}.
$$
The following result, which is new, is an easy consequence of Theorem \ref{THMconjLSC}. 

\begin{theorem}
\label{THMRiem2}
The function $t_{conj}^x$ is locally semiconcave on its domain which is an open subset of $\mathcal{S}_1^x$.
\end{theorem}

We mention that Itoh and Tanaka proved in \cite{it01} the locally Lipschitz regularity of the distance function to the conjugate locus from a point. 

\subsection{}

Let $(M,g)$ be a complete smooth Riemannian manifold. For every $x\in M$, we call \textit{tangent nonfocal domain} of $x$ the subset of $T_xM$ defined by
$$
\mathcal{NF}(x) := \left\{ t v \ \vert \ \|v\|_x=1,  0 \leq t \leq t_{conj}^x (v) \right\}.
$$
By Theorem \ref{THMRiem2}, we know that for every $x\in M$, the set $\mathcal{NF}(x)$ is an open subset of $T_xM$ whose the boundary is given by the "graph" of the function $t_{conj}^x$ which is locally semiconcave on its domain. We call \textit{$C^4$-deformation} of the round sphere $(\S^n,g^{can})$ any Riemannian manifold of the form $(M,g^{\e})$ with $M=\S^n$ and $g^{\e}$ close to $g$ in $C^4$-topology. The strategy that we develop to prove Theorem \ref{THMconjLSC} allows to prove the following result.

\begin{theorem}
\label{THMdeformation}
If $(M,g)$ is a $C^4$-deformation of the round sphere $(\S^n,g^{can})$, then for every $x\in M$, the set $\mathcal{NF}(x)$ is strictly uniformly convex.
\end{theorem}

We provide the proof of this result in the next section.

\subsection{Proof of Theorem \ref{THMdeformation}}

Consider the stereographic projection of the sphere $\S^n \subset \R^{n+1}$ centered at the origin and
of radius $1$ from the north pole onto the space $\R^n \simeq \R^n \times \{0\} \subset \R^{n+1}$.  This is the map $\sigma : \S^n \setminus \{N\} \rightarrow \R^n$ that sends a point $X\in \S^n \setminus \{N\} \subset \R^{n+1}$, written $X=(x,\lambda)$ with $x=(x_1,\cdots, x_n) \in \R^n$ and $\lambda \in \R$, to $y\in \R^n$, where $Y:=(y,0)$ is the point where the line through $N$ and $P$ intersects the hyperplane $\{\lambda=0\}$ in $\R^{n+1}$. That is,
$$
\sigma (X) =  \frac{x}{1-\lambda} \qquad \forall\, X=(x,\lambda) \in \S^n \setminus \{N\} \subset \R^{n+1}.
$$
The function $\sigma$ is a smooth diffeomorphism from $\S^n \setminus \{N\}$ onto $\R^n$. Its inverse is given by
$$
\sigma^{-1} (y) =\left( \frac{2y}{1+|y|^2},  \frac{|y|^2 -1}{1+|y|^2} \right) \qquad \forall\, y \in \R^n,
$$
where $|\cdot|$ denotes the Euclidean norm on $\R^n$. The pushforward of the round metric on $\S^n$ is given by
$$
g_y (v,v) = \frac{4}{(1+|y|^2)^2} |v|^2 \qquad \forall\, y,v \in \R^n.
$$
The metric $g$ is conformal to the Euclidean metric $g^{eucl}(\cdot,\cdot) = \langle \cdot, \cdot \rangle$, that is it satisfies $g=e^{2f}g^{eucl}$ with $f(y)= \log (2) -\log(1+|y|^2)$. Hence the Riemannian connection associated to $g$ is given by 
\begin{eqnarray}
\label{conform}
\nabla^g_V W & = & \nabla^{eucl}_V W + df(V)W + df(W)V - g^{eucl} (V,W) \nabla f.
\end{eqnarray}

Set $\bar{X}:=(\bar{y},0) \in \S^n$ with $\bar{y}:=(-1,0,\cdots ,0)\in \R^n$ and $\bar{V}:=(\bar{v},-1)$ with $\bar{v}:=0\in \R^n$. For each vector $V =(0,v)=(0,v_1,\cdots ,v_n) \in \R^{n+1}$ such that $|V|=|v|=1$ and $|V-\bar{V}| < 1$, the minimizing geodesic on the sphere starting from $\bar{X}$ with initial speed $V$ is given by
$$
\gamma_V(t) = \cos (t) \bar{X}  + \sin(t) V \qquad \forall t \in [0,\pi].
$$
Its projection by stereographic projection is given by 
$$
\theta_V(t) := \sigma \left( \gamma_V(t)\right) = \left( \frac{- \cos(t)}{1- \sin (t) v_n},  \frac{\sin(t) v_1}{1- \sin (t) v_n}, \cdots ,  \frac{\sin(t) v_{n-1}}{1- \sin (t) v_n}  \right).
$$
Therefore, $\theta_V$ is the geodesic starting from $\sigma(\bar{X})=\bar{y}$ with initial speed $v=d\sigma(\bar{X})(V)=: \sigma_*(V)$ in $\R^2$ equipped with the Riemannian metric $g$.  For every $V$ as above, one has 
$$
z^V = (z_1^V, \cdots ,z_n^V) := \theta_V (\pi /2) = \left( 0, \frac{v_1}{1-v_n} , \cdots ,\frac{v_{n-1}}{1-v_n}\right).
$$
There are contained in the hyperplan
$$
S := \left\{ y =(y_1,\cdots ,y_n) \in \R^n \ \vert \ y_1=0\right\}.
$$
Set $\mathcal{V}:= \left\{ V=(0,v) \in \R^{n+1} \ \vert \  |V|=1, \, |V-\bar{V}| < 1\right\}$ and define the mapping $\mathcal{Z} : \mathcal{V} \rightarrow S$ by,
$$
\mathcal{Z} (V) := z^V \qquad \forall V \in \mathcal{V}.
$$
This mapping is one-to-one from $\mathcal{V}$ into its image $\mathcal{S}:=\mathcal{Z} (\mathcal{V}) \subset S$ ; its inverse is given by 
$$
\mathcal{Z}^{-1} (z) = \left(0, \frac{2z_2}{1+|z|^2}, \cdots , \frac{2z_n}{1+|z|^2}, \frac{|z|^2-1}{1+|z|^2}\right).
$$
In particular, we note that for every $V = (0,v_1,\cdots ,v_n) \in \mathcal{V}$, one has 
\begin{eqnarray}
\label{eqv_n}
1+\bigl|z^V\bigr|^2 = \frac{2}{1-v_n}.
\end{eqnarray} 

Let $H:\R^n \times \R^n \rightarrow \R$ be the Hamiltonian canonically associated to the metric $g$, that is, 
$$
H(y,p) = \frac{(1+|y|^2)^2}{8} |p|^2 \qquad \forall\, y,p \in \R^n.
$$
The Hamiltonian system associated to $H$ is given by
\begin{eqnarray}
\label{H99}
\left\{ \begin{array}{ccccl}
\dot y &=& \frac{\partial H}{\partial p}(y,p) &=& \frac{(1+|y|^2)^2}{4} p\\
\dot p&=&-\frac{\partial H}{\partial y}(y,p) &=&-\frac{(1+|y|^2)|p|^2}{2}y.
\end{array}
\right.
\end{eqnarray}
For every $V\in \mathcal{V}$ the solution $(y^V,p^V)$  of (\ref{H99}) starting at $\bigl(\bar{y},p^V(0)=(-v_n,v_1,\cdots ,v_{n-1})\bigr)$ is given by 
$$
\left\{ \begin{array}{rcl}
y^V(t) &= & \theta_V(t) \\
p^V(t) & = & \frac{4\dot{\theta}_V(t)}{\left( 1 + |\theta_V(t)|^2\right)^2}  =  \bigl(\sin(t)-v_n, \cos(t) v_1, \cdots ,\cos(t) v_{n-1} \bigr).
\end{array}
\right.
$$
Set for every $z=(0,z_{n-1}) \in \mathcal{S}$,
$$
\exp (z,s) := \pi \left(  \phi_s^H (z, P(z))\right),
$$
where $P(z)$ is defined by 
$$
P(z) := p^{\mathcal{Z}^{-1}(z)}(\pi/2) = \left(\frac{2}{1+|z|^2},0,\cdots ,0\right).
$$
We denote by $t_{conj}(z)$ the first time $t\geq 0$ such that the mapping $z\mapsto \exp (z,t)$ is singular. The linearized Hamiltonian system along a given solution $(y(t),p(t))$ of (\ref{H99}) is given by 
$$
\left\{ \begin{array}{ccl}
\dot h &=& (1+|y|^2)\langle y, h\rangle p + \frac{(1+|y|^2)^2}{4}q\\
\dot q &=& - \frac{(1+|y|^2)^2|p|^2}{2} h - |p|^2\langle y, h\rangle y - (1+|y|^2)(p\cdot q) y
\end{array}
\right.
$$
We note that $h$ is a Jacobi vector field along the geodesic $t\mapsto y(t)$. As in Lemma \ref{lem:conj1}, we set for every $z\in \mathcal{S}$ and every $s>0$,
$$
J(z,s) := \left\{ \left( \begin{array}{c}
h \\
q
\end{array}
\right) \ \vert \ \phi_s^H (h,q) \in \{0\} \times \R^n \right\},
$$
and we denote by $K(z,s)$ the $n\times n$ symmetric matrix such that 
$$
J(z,s) =  \left\{ \left( \begin{array}{c} h\\ K(z,s)h \end{array}
\right)  \ \vert \ h \in \R^n \right\}.
$$
Let us now compute the mapping $(z,s) \mapsto J(z,s)$.\\

Let $z\in \mathcal{S}$ be fixed and $V=(0,v_1,\cdots ,v_n) \in \mathcal{V}$ be such that $\mathcal{Z}(V)=z$. Set for every $s \geq 0$,
\begin{multline*}
E_1^z(s) := \dot{\theta}_{V}(s+\pi/2) \\
= \left( \frac{\sin (s+\pi/2) - v_n}{( 1- \sin (s+\pi/2)v_n)^2}, \frac{\cos (s+\pi/2)v_1}{( 1- \sin (s+\pi/2)v_n)^2},\cdots ,  \frac{\cos (s+\pi/2)v_{n-1}}{( 1- \sin (s+\pi/2)v_n)^2} \right).
\end{multline*}
Denote by $\{e_1,\cdots ,e_n\}$ the canonical basis of $\R^n$. One check easily that 
\begin{eqnarray}
\label{E1}
E_1^z(0)= \frac{1}{1-v_n}e_1.
\end{eqnarray}
Let $E_2^z(s), \cdots ,E_n^z(s)$ be $(n-1)$ vectors along the curve $\theta_s :s \mapsto \exp (z,s)$ satisfying
\begin{eqnarray}
\label{E2}
E_i^z(0) =e_i \qquad \forall i=2, \cdots ,n,
\end{eqnarray}
and such that $E_1^z,\cdots ,E_n^z$  form a basis of parallel vector fields along $\theta_{z}$. One has 
\begin{eqnarray}
\label{E3}
\dot{E}_1^z(0) = \left(0,\frac{-v_1}{(1-v_n)^2}, \cdots , \frac{-v_{n-1}}{(1-v_n)^2} \right).
\end{eqnarray}
Moreover, thanks to (\ref{conform}), one has
\begin{eqnarray}
\dot{E}_i^z (0) = \frac{v_{i-1}}{1-v_n} e_1 \qquad \forall i=2,\cdots, n.
\end{eqnarray}
Let $(h,q)$ be a solution of the linearized Hamiltonian system along $\theta_{V}$ such that $h(\bar t)=0$ for some $\bar{t}>0$.
Since $E_1^z(t), \cdots ,E_n^z(t)$ form a basis of parallel vector fields along $\theta_{V}$, there are $n$ smooth functions $u_{1}, \cdots , u_n $ such that
\begin{eqnarray}
\label{eqh}
h(t) = \sum_{i=1}^n u_i(t) E_i^z(t) \qquad \forall\, t.
\end{eqnarray}
Hence, since $h$ is a Jacobi vector field along $\theta_{V}$, its second covariant derivative along $\theta_{V}$ is given by
$$
D_t^2 h(t) = \sum_{i=1}^n  \ddot{u}_i(t) E_i^z(t).
$$
Therefore, since $(\R^n,g)$ has constant curvature,  one has
\begin{eqnarray*}
0&=& D_t^2 h + R(h,\dot{\theta}_{V}) \dot{\theta}_{V}  \\
& =& D_t^2 h + g \left(\dot{\theta}_{V},\dot{\theta}_{V}\right) h - g \left(h, \dot{\theta}_{V}\right) \dot{\theta}_{V} \\
 & =& \sum_{i=1}^n \ddot{u}_i(t) E_i^z(t) + \sum_{i=1}^n  u_i(t) E_i^z(t)  - u_1(t) \dot{\theta}_{V}(t) \\
 & =&  \ddot{u}_1(t) E_1^z(t) + \sum_{i=2}^n [\ddot{u}_i(t)+ u_i(t)] E_i^z(t).
\end{eqnarray*}
We deduce that there are $2n$ constants $\lambda_1^i,\lambda_2^i$ with $i=1,\cdots ,n$ such that
$$
\left\{
\begin{array}{rcl}
u_1(t) & = & \lambda_1^1 + \lambda_2^1 ( t-\pi/2) \\
u_i(t) & = & \lambda_1^i \cos (t) + \lambda_2^i \sin (t) \quad \forall i=2,\cdots ,n.
\end{array}
\right.
$$
Moreover, since $h(\bar{t})=0$, one has $u_i(\bar{t})=0$ for all $i$, which yields
$$
\lambda_2^1= -\frac{\lambda_1^1}{\bar{t}-\pi /2} \quad \mbox{ and } \quad \lambda_1^i = -\lambda_2^i \frac{\sin(\bar{t})}{\cos(\bar{t})} \quad \forall i=2,\cdots ,n.
$$
By (\ref{E1}), (\ref{E2}), Since $E_1^z(\pi/2) =  \frac{1}{1-v_n}e_1$ and $E_i^z(\pi/2) =e_i$ for any $i=2,\cdots ,n$,  (\ref{eqh}) yields
$$
h_1(\pi/2) = \frac{\lambda_1^1}{1-v_n}, \qquad 
h_i(\pi/2) = \lambda_2^i \quad \forall i=2,\cdots ,n.
$$
Furthermore, differentiating $h(t) = \sum_{i=1}^n u_i(t) E_i^z(t)$ at $t=\pi/2$, we obtain
\begin{eqnarray*}
\dot{h}_1(\pi/2) & = & \frac{\lambda_2^1}{1-v_n}  + \sum_{i=2}^n \lambda_2^i  \frac{v_{i-1}}{1-v_n} \\
& = & -\frac{\lambda_1^1}{(\bar{t}-\pi/2)(1-v_n)}  + \sum_{i=2}^n \frac{v_{i-1}}{1-v_n} h_i(\pi/2) \\
& = & - \frac{1}{t-\pi/2} h_1(\pi/2) + \sum_{i=2}^n \frac{v_{i-1}}{1-v_n} h_i(\pi/2),
\end{eqnarray*}
and for every $i=2,\cdots,n$,
\begin{eqnarray*}
\dot{h}_i(\pi/2)  & = & -\lambda_1^i - \frac{\lambda_1^1 v_{i-1}}{(1-v_n)^2} \\
& = & \frac{\sin(t)}{\cos(t)} \lambda_2^i - \frac{v_{i-1}}{1-v_n} h_1(\pi/2) \\
& = & \frac{\sin(t)}{\cos(t)} h_i(\pi/2) - \frac{v_{i-1}}{1-v_n} h_1(\pi/2).
\end{eqnarray*}
But one has 
$$
\theta_V (\pi /2) = \left( 0, \frac{v_1}{1-v_n} , \cdots ,\frac{v_{n-1}}{1-v_n}\right),
$$
and 
$$
P_V(\pi/2) = (1-v_n) e_1.
$$
From the linearized Hamiltonian system, one has
$$
q(\pi/2) =  \frac{4}{(1+|z|^2)^2} \dot{h}(\pi/2)  - \frac{4 \langle h(\pi/2), z\rangle}{1+|z|^2} P_V(\pi/2).
$$
Thus we finally obtain that for every $z\in \mathcal{S}$ and any $s\in [0,\pi)$, one has 
$$
K(z,s) = \frac{-4}{\left(1+|z|^2\right)^2}\left( \begin{array}{ccccc}
 1/s & z_2 & \cdots & \cdots & z_n \\
z_2 &  - \frac{\cos(s)}{\sin(s)} & 0 & \cdots & 0\\
z_3 & 0 &  - \frac{\cos(s)}{\sin(s)}  & \cdots & 0 \\
\vdots & \vdots & \ddots & & \vdots \\
z_n & 0 & \cdots & \cdots & - \frac{\cos(s)}{\sin(s)}  
\end{array}
\right).
$$

Let $z\in \mathcal{S}$ be fixed, let us compute $U(z)$. One has $P(z)=(2/(1+|z|^2),0,\cdots ,0)$. Hence one has 
$$
U(z) = \frac{-4}{\left(1+|z|^2\right)^2} \left( \begin{array}{cccc}
0 & z_2 & \cdots & z_n \\
z_2 & 0 &\cdots & 0 \\
\vdots & \vdots & \ddots & \vdots \\
z_n & 0 & \cdots & 0 
\end{array}
\right)
$$

Therefore we deduce that for any $z\in \mathcal{S}$ and $s\in [0,\pi)$, the symmetric matrix $K(z,s)-U(z)$ is given by
\begin{eqnarray}
\label{FORMULAfin}
K(z,s)-U(z) = \frac{4}{\left(1+|z|^2\right)^2}\left( \begin{array}{ccccc}
- 1/s &  0 & \cdots & \cdots & 0 \\
0 &  \frac{\cos(s)}{\sin(s)} & 0 & \cdots & 0\\
0 & 0 &   \frac{\cos(s)}{\sin(s)}  & \cdots & 0 \\
\vdots & \vdots & \ddots & & \vdots \\
0 & 0 & \cdots & \cdots &  \frac{\cos(s)}{\sin(s)}  
\end{array}
\right).
\end{eqnarray}
Moreover, recalling that $t_{conj}:\mathcal{S} \rightarrow \R$ denotes the distance function to the conjugate locus associated with the Dirichlet-type Hamilton-Jacobi equation
\begin{eqnarray*}
\left\{ \begin{array}{rl}
H(x,du(x)) -1/2 = 0, & \qquad \forall x \in \Omega,\\
u(x) = 0, & \qquad \forall x \in \partial \mathcal{S}
\end{array} \right.
\end{eqnarray*}
(where $\Omega$ is an open neighborhood along the geodesic $\theta_{\bar{V}}(\cdot+\pi/2)$), we have 
$$
t_{conj}^{\bar{y}} \bigl( \sigma_* (V)\bigr) = t_{conj} \bigl( \mathcal{Z}(v)\bigr) + \frac{\pi}{2} \qquad \forall V \in \mathcal{V}.
$$

Let us now consider a smooth metric $g^{\epsilon}$ on the sphere $\S^n$ and $x\in \S^n$. By symmetry, we may assume that $x=\hat{X}$. By Proposition \ref{PROPfin}, there is a constant $K>0$ such that, if for any $v \in T_x \S^n$ with $\|v\|_x^{\epsilon}=1$ (here $\|\cdot \|_x^{\epsilon}$ denotes the norm in $T_x\S^n$ associated with $g^{\epsilon}$), there is a function $\tau_v$ of class $C^2$ defined on the unit sphere in $T_x\S^n$ associated with $g^{\epsilon}$ such that
$$
t_{conj}^x(v) = \tau_v(v), \quad t_{conj}^x \leq \tau_v \quad \mbox{ and } \|D^2\tau_v \|_{\infty} < K,
$$
then the set $\mathcal{NF}(x)$ is strict uniformly convex. Let $v\in T_x\S^n$ with $\|v\|_x^{\epsilon}=1$, again by symmetry, we may assume that $v$ is close to $\bar{V}$. Using the stereographic projection as above, we can push the new metric $g'$ into a metric $\tilde{g}$ on $\R^2$ and $v$ into a speed $\tilde{v}$. Thus, we have to show that there is a $C^2$  function $\tau : \tilde{\mathcal{S}}_{\bar{y}}^1 \rightarrow \R$  (where $ \tilde{\mathcal{S}}_{\bar{y}}^1$ denotes the unit sphere at $\bar{y}$ with respect to $\tilde{g}$) such that 
$$
\tilde{t}_{conj}^{\bar{y}} (\tilde{v}) = \tau(\tilde{v}), \quad \tilde{t}_{conj}^{\bar{y}} \leq \tau \quad \mbox{ and } \|D^2 \tilde{\tau} \|_{\infty} < K.
$$ 
For every $v\in  \tilde{\mathcal{S}}_{\bar{y}}^1$, we denote by $\tilde{\theta}_v$ the geodesic (with respect to $\tilde{g}$) starting at $\bar{y}$ with initial speed $v$. Let $\tilde{\mathcal{V}}$ be an open neighborhood of $\tilde{v}$ in  $ \tilde{\mathcal{S}}_{\bar{y}}^1$, set 
$$
\tilde{\mathcal{Z}}(v) :=   \tilde{\theta}_v(\pi/2) \quad \mbox{ and } \quad \tilde{\mathcal{S}} := \left\{ \tilde{\theta}_v(\pi/2) \ \vert \ v \in \tilde{V} \right\}.
$$
As above, if we denote by $\tilde{t}_{conj}$ the distance function to the conjugate locus associated with  the Dirichlet-type Hamilton-Jacobi equation
\begin{eqnarray}
\label{HJfin}
\left\{ \begin{array}{rl}
\tilde{H}(x,du(x)) - 1/2 = 0, & \qquad \forall x \in \Omega,\\
u(x) = 0, & \qquad \forall x \in \partial \tilde{\mathcal{S}}
\end{array} \right.
\end{eqnarray}
(where $\tilde{H}$ denotes the Hamiltonian which is canonically associated with $\tilde{g}$), we have 
$$
\tilde{t}_{conj}^{\bar{y}} ( v ) = \tilde{t}_{conj} \bigl( \tilde{\mathcal{Z}}(v)\bigr) + \frac{\pi}{2} \qquad \forall v \in \tilde{\mathcal{V}}.
$$
Set $\tilde{z}:=\tilde{\mathcal{Z}} (\tilde{v})$. Therefore, we have to show that there is a function $\tau : \tilde{\mathcal{S}} \rightarrow \R$ of class $C^2$ such that
 $$
\tilde{t}_{conj} (\tilde{z}) = \tau(\tilde{z}), \quad \tilde{t}_{conj} \leq \tau \quad \mbox{ and } \|D^2 \tilde{\tau} \|_{\infty} \mbox{ small enough}.
$$ 
Denote by $\tilde{K}$ and $\tilde{U}$ the functions associated with (\ref{HJfin}) which have been defined in Section 2. Let $\tilde{s}>0$ and $h\in \R^n$ with $\|h\|=1$ be such that $d\exp (\tilde{z},\tilde{s})(h)=0$. By Lemma  \label{lem:conj1}, this means that 
$$
\langle \bigl[ \tilde{K}(\tilde{z},\tilde{s}) - U(\tilde{z})\bigr] h,h\rangle =0.
$$
As in the proof of Theorems \ref{THMconjLIP} and \ref{THMconjLSC} , we define a function $\tilde{\Psi}$ in a neighborhood of $(\tilde{z},\tilde{s})$ by 
$$
\tilde{\Psi} (z,s) := \langle \bigr[ \tilde{K}(z,s) - U(z)\bigr] h,h\rangle.
$$
As above, the Implicit Function Theorem will provide a function $\tilde{\tau}$ defined in a neighborhood of $\tilde{z}$ such that 
$$
\tilde{\Psi} (z,\tilde{\tau}(z))=0 \qquad \forall z.
$$
Using (\ref{FORMULAfin}), we define the function $\Psi$ in a neighborhood of $(\bar{z},\pi/2)$ by
$$
\Psi (z,s) := \langle [ K(z,s)-U(z)] h,h\rangle.
$$
If the metric $g^{\epsilon}$ is close to the metric $g^{can}$ on $\S^n$ for the $C^4$ topology, then the function $\tilde{Psi}$ (which depends upon $g^{\epsilon}$) will $C^2$ close (up to a change of variables between $\mathcal{S}$ and $\tilde{\mathcal{S}}$) to the function $\Psi$. Using the fact that the first and second derivatives in the $z$ variable of $\Psi$ vanish at  time $\pi/2$, we leave the reader to conclude that the function $\tilde{\tau}$ provided by the Implicit Function Theorem is flat enough. This concludes the proof of Theorem  \ref{THMdeformation}.

\section{Comments}

\subsection{} In dimension $2$, the mapping $t_{conj}$ can be shown to be of class $C^{k-2,1}$ on its domain.

\subsection{} The proof of Theorem  \ref{THMcut} (see first case in its proof) shows that, if the datas are of class at least $C^{3,1}$, then the function $t_{cut}$ is locally semiconcave on any open set $S\subset \mathcal{S}_1^x$ satisfying 
$$
\diam \bigl( \partial d_g^x (\exp_x(t_{cut}^x (v)v) ) \bigr) >0 \qquad \forall v \in S.
$$
This kind of result has been used by Loeper and Villani \cite{lv08} in the context of optimal transportation theory. We mention that, given a general smooth compact Riemannian manifold, we do not know if the functions $t_{cut}^x$ are locally semiconcave on $\mathcal{S}_x^1$. 

\subsection{} Our result concerning the strict uniform convexity of nonfocal domains for small deformation of the round spheres is motivated by regularity issues in optimal transportation theory, see \cite{fr08,frv}.

\subsection{} In the present paper, we deduce Theorem \ref{THMdeformation} as a corollary of our results concerning viscosity solutions of Hamiltonian-Jacobi equations. In other terms, we used the symplectic viewpoint. We mention that Theorem \ref{THMdeformation} could as well be obtained with a purely Riemannian approach using some special properties of Jacobi fields, see \cite[Chapter 14, Third Appendix]{villaniSF}.

\appendix 

\section{Strictly uniformly convex sets}

Let $n\geq 2$ be fixed; in the sequel, if $A$ is a given subset of $\R^n$, we denote by $d(\cdot,A)$ the distance function to $A$. Following \cite[Appendix B]{lv08}, a natural notion of uniformly convex set is given by the following: 

\begin{definition}
\label{DEFconvex}
A compact set $A\subset \R^n$ is said to be strictly uniformly convex if there is $\kappa>0$ such that
\begin{eqnarray}
\label{EQconvex}
d(\lambda x + (1-\lambda) y,\partial A) \geq \kappa \lambda (1-\lambda) |x-y|^2.
\end{eqnarray}
\end{definition}

The following proposition more or less well-known gives a local characterization of strictly uniformly convex sets. We refer the reader to \cite[Appendix B]{lv08} for its proof.

\begin{proposition}
Let $A$ be a compact subset of $\R^n$ which Lipschitz boundary. Then the two following properties are equivalent: 
\begin{itemize}
\item[(i)] $A$ is strictly uniformly convex; 
\item[(ii)] there is $\kappa >0$ such that for every $x\in \partial A$, there are $\delta_x >0$  and $z_x\in \R^n$ with $|z_x-x| = 1/\kappa$ satisfying 
$$
A \cap B(x,\delta_x) \subset B(z_x,1/\kappa).
$$
\end{itemize}
\end{proposition}

As a corollary, one has the following result. 

\begin{proposition}
\label{PROPfin}
Let $\mathcal{T}: \S^{n-1} \rightarrow \R$ be a Lipschitz function, set 
$$
A_{\mathcal{T}}:= \left\{ t \mathcal{T}(v) v \ \vert \ v \in \S^{n-1}, \, t\in [0,1]  \right\} \subset \R^n.
$$
There is $K > 0 $ such that if, for every $v\in \S^{n-1}$, there is a function $\tau :\S^{n-1} \rightarrow \R$ of class $C^2$ satisfying $\tau(v)=\mathcal{T}(v), \mathcal{T}\leq \tau$ and $\|D^2\tau\|_{\infty} \leq K$, then the set $A_{\mathcal{T}}$ is strictly uniformly convex.
\end{proposition}

\end{document}